\documentclass{article}
\usepackage{amsmath,amssymb,amsthm}
\usepackage{amsfonts}
\usepackage{enumerate}
\usepackage[single]{accents}
\usepackage[margin=1.5in]{geometry}
\usepackage{color}
\usepackage{graphicx}
\usepackage{float}
\usepackage{tikz}
\usepackage{authblk}
\usetikzlibrary{shapes}
\usetikzlibrary{plotmarks}
\usetikzlibrary{arrows}
\usetikzlibrary{positioning}
\definecolor{c0000FF}{RGB}{0,0,255}
\def\f#1#2{\mathfrak{Fr}_{#1}#2}
\def\t#1#2{\mathfrak{Tm}_{#1}#2}
\def\mod#1#2{(\mathfrak{Cm}(S^{\tau,K}),e^{\tau},#1)\models #2}
\def\a#1{\mathfrak{#1}}

\def\qut#1{``#1''}
\title{Weak G\"odel's incompleteness property  for some decidable versions of first order logic}
\author{Mohamed Khaled\thanks{rutmohamed@yahoo.com}}
\affil{Department of Mathematics and its applications, Central European University, Budapest, Hungary}
\affil{Department of Mathematics, Faculty of Science, Cairo University, Giza, Egypt}
\date{}
\newtheorem{definition}{Defintion}[section]
\newtheorem{theorem}{Theorem}[section]

\newtheorem{lemma}{Lemma}[section]
\begin{document}
\maketitle
\begin{abstract}
The founding of the theory of cylindric algebras, by Alfred Tarski, was a conscious effort to create algebras out of first order predicate calculus. Let $n\in\omega$. The classes of non-commutative cylindric algebras ($NCA_n$)
and weakened cylindric algebras ($WCA_n$) were shown, by Istv\'an N\'emeti, to be examples of decidable versions of first order logic with $n$ variables. In this article, we give new proofs for the decidability of the equational theories of these classes. We also give an answer to the open problem, posed by N\'emeti in 1985, addressing the atomicity of the finitely generated free algebras of these classes. We prove that all the finitely generated free algebras of the  varieties $NCA_n$ and $WCA_n$ are not atomic. In other words, we prove that the corresponding versions of first order logic have weak G\"odel's incompleteness property.
\end{abstract}
\section{Introduction}
The process which is called algebraization of logics started by the English mathematician George Boole.
Indeed, he introduced and started to investigate the class of Boolean algebras which is directly related to the
development of classical propositional logic. After then a continuous flow of steps in this direction has been
accomplished. The calculus of relations was created and developed by De Morgan, Peirce and Schr\"oder as a result of
the continues efforts searching for \qut{good general algebra of logic}. These efforts took place decades before the
emergence of first order calculus. The early notation for quantifiers originates with Peirce. The original version of
the L\"owenheim-Skolem theorem, c.f. \cite{low15}, is not a theorem about first order logic but about the calculus of
relations. Thus, first order predicate calculus has its origins in algebraic logic.

Algebraic logic is concerned with the ways of algebraizing logics and with the ways of investigating the algebras of
logics. The framework of algebraic logic is universal algebra. Universal algebra is the field which investigates
classes of algebras in general, interconnections, fundamental properties and so on. In other words, universal algebra
is a unifying framework that can provide a plan for investigating certain properties of the algebras of logics.
As in Boolean algebras, the algebras of logics often provide very general kind of geometry associated with basic
set-theoretic notions. Therefore, in algebraic logic one also have the advantage of illustrating the different
concepts which gives better understanding in a sense. Here, we are interested in some variations of cylindric
algebras. Cylindric algebras were introduced by Alfred Tarski, around 1950, as the algebras of first order predicate
calculus.

One of the main interests in the filed of algebraic logic is to the study the variants of the algebras of logics.
The variants of cylindric algebras correspond to fragments of first order logic. We are interested in two variants of
cylindric algebras: non-commutative cylindric algebras and weakened cylindric algebras. The class of non-commutative
cylindric algebras was first introduced and investigated by Richard Thompson. The class of weakened cylindric algebras
was introduced by Istv\'an N\'emeti. In \cite{nem86}, N\'emeti showed that these classes have decidable equational
theories. These algebras correspond to fragments of first order logic in which quantifiers don't need to permute.
Thus, N\'emeti showed that it is the permutability of quantifiers which is responsible for the undecidability of
first order logic. We note that N\'emeti's results answered some question that came years later. Indeed,
Johan van Benthem in 1994 \cite{vB95} asked the question \qut{What would have to be weakened in standard predicate
logic to get a decidable version}.

We concentrate on the finite dimensional algebras. Let $n\geq 2$ be finite. We give new proofs for the decidability of
the equational theories of the classes of non commutative ($NCA_n$) and weakened ($WCA_n$) cylindric algebras of
dimension $n$. We also show that these classes are generated by their finite members in the sense that their equational
theories coincide with the equational theories of their finite members. Furthermore, we study the atomicity of the
finitely generated free algebras of these classes. The atomicity problem of these free algebras goes back to 1985
when Istvan N\'emeti posed it in his Academic Doctoral Dissertation \cite{nem86}. In 1991, N\'emeti posed the same
problem again in \cite[Problem 38]{al}. This problem was posed again as an open problem in 2013 in the most recent
book in algebraic logic \cite[Problem 1.3.3]{ca4}. We prove that all the  finitely generated free algebras of the
classes $NCA_n$ and $WCA_n$ are not atomic. The non-atomicity of the free algebras of logics is equivalent to
weak G\"odel's incompleteness property of the corresponding logic. See \cite[proposition8]{prenem85} and \cite{gyn}.
\newpage
\section{Non-commutative variants of cylindric algebras}
Fix a finite number $n\geq 2$. The algebraic type $cyl_n$ has constant symbols $0,1,d_{ij}$ ($i,j\in n$), unary function symbols $-,c_i$ ($i\in n$) and binary function symbols $\cdot,+$. Denote the class of algebras of type $cyl_{n}$ by $CTA_{n}$. Define the following sets of equations in the type $cyl_{n}$. \begin{eqnarray*}
C_0 &=& \{\text{The equations characterizing Boolean algebras for $+,\cdot,-,0,1$}\}.\\
C_1 &=& \{c_i0=0 : i\in n\}.\\
C_2 &=& \{x\leq c_ix : i\in n\}.\\
C_3 &=& \{c_i(x\cdot c_iy)=c_ix\cdot c_iy : i\in n\}.\\
C_4 &=& \{c_ic_jx=c_jc_ix : i,j\in n\}.\\
C_5 &=& \{d_{ii}=1 : i\in n\}.\\
C_6 &=& \{d_{ij}=c_k(d_{ik}\cdot d_{kj}): i,j,k\in n, k\not\in\{i,j\}\}.\\
WC_6 &=& \{d_{ik}\cdot d_{kj}\leq d_{ij}=d_{ji}=c_kd_{ji}:i,j,k\in n,k\not\in\{i,j\}\}.\\
C_7 &=& \{c_i(d_{ij}\cdot x)\cdot c_i(d_{ij}\cdot -x)=0 : i,j\in n, i\not=j\}.
\end{eqnarray*}
The class of cylindric algebras of dimension $n$ is defined to be the subclass of $CTA_n$ characterized by the above equations (Note that $C_6$ is stronger than $WC_6$).
$$CA_n=\{\a{A}\in CTA_n:\a{A}\models\{C_0,C_1,C_2,C_3,C_4,C_5,C_6,C_7\}\}.$$
The class of non commutative cylindric algebras, of dimension $n$, is defined by realizing the commutativity axioms.
$$NCA_n=\{\a{A}\in CTA_n:\a{A}\models\{C_0,C_1,C_2,C_3,C_5,C_6,C_7\}\}.$$
The class of weakened cylindric algebras, of dimension $n$, has the same characterization of the class $NCA_n$ except that $C_6$ is replaced by the weaker version $WC_6$.
$$WCA_n=\{\a{A}\in CTA_n:\a{A}\models\{C_0,C_1,C_2,C_3,C_5,C_7\}\cup\{WC_6\}\}.$$

In \cite{nem86} ( and also in \cite{nem95}), it was shown that the classes $NCA_n$ and $WCA_n$ are generated by their finite members and
that they have decidable equational theories. Then Nemeti asked whether their finitely generated free algebras
are atomic or not. This problem presents some difficulties from an algebraic point of view.
In \cite{ca1}, Henkin proved that the finitely generated free algebras of the class $CA_2$ are atomic.
Hajnal Andr\'eka and Istv\'an N\'emeti figured out that Henkin's proof depends on the fact that $CA_2$ is a
discriminator variety that is generated by its finite members. But, however they are generated by their finite members,
N\'emeti has proved that non of the classes $NCA_n$ and $WCA_n$ is a discriminator variety.

Let $K\in\{NCA_n, WCA_n\}$. We give new proof for the fact that the equational theory of $K$ is the same as the equational theory of its finite members. Toward that, we construct for every satisfiable term a finite structure (in $K$) that witnesses the satisfiability of this term. Then we conclude that the equational theory of the class $K$ is decidable. Moreover, we use this finite structures to investigate the atomicity of the finitely generated free algebras of the variety $K$. Fix a finite cardinal $m\in\omega$. $\t{m,cyl_n}{}$ denotes the term algebra of type $cyl_n$ generated by
$m$-many free variables. $\f{m}{K}$ denotes the free algebra of the
class $K$ generated by $m$-many generators.
\subsection{Disjunctive normal forms}
we reduce the problem by using disjunctive normal forms in the language of $CTA_n$. Disjunctive normal forms can
provide elegant and constructive proofs of many standard results, c.f., \cite{anderson} and \cite{fine}.
Let $\prod,\sum$ be the grouped versions of $\cdot,+$, respectively. Let $T\subseteq\t{m,cyl_n}{}$ be a finite set
of terms and let $\alpha\in{^T\{-1,1\}}$. Let,
\begin{center}$CT:=\{c_i\tau:i<n, \tau\in T\}$ \text{ } \text{ } and \text{ } \text{ } $T^{\alpha}:=\prod\{\tau^{\alpha}:\tau\in T\}$.\end{center}
Where, for every $\tau\in T$, $\tau^{\alpha}=\tau$ if $\alpha(\tau)=1$ and $\tau^{\alpha}=-\tau$ otherwise. Now,
for every $k\in\omega$, we define a set $F^{n,m}_k\subseteq\t{m,cyl_n}{}$ of normal forms of degree $k$ such that
every normal form contains complete information about the cylindrifications of the normal forms of the first smaller
degree.
\begin{definition}
Set $D_{n,m}=\{d_{ij}:i,j<n\}\cup\{x_0,\ldots,x_{m-1}\}$, where $x_0,\ldots,x_{m-1}$ are the $m$ free
variables that generate $\t{m,cyl_n}{}$. For every $k\in\omega$, we define the followings inductively.
\begin{enumerate}[-]
\item The normal forms of degree $0$, $F^{n,m}_{0}=\{D_{n,m}^{\beta}:\beta\in{^{D_{n,m}}\{-1,1\}}\}$.
\item The set of normal forms of degree $k+1$, \begin{center}$F^{n,m}_{k+1}=\{D_{n,m}^{\beta}\cdot (CF^{n,m}_k)^{\alpha}:\beta\in{^{D_{n,m}}\{-1,1\}}\text{ and }\alpha\in{^{CF^{n,m}_k}\{-1,1\}}\}$.\end{center}
\item The set of all forms, $F^{n,m}=\bigcup_{k\in\omega} F^{n,m}_k$.
\end{enumerate}
\end{definition}
Let $K^{'}$ be the class of all Boolean algebras with operators of type $cyl_n$. The following theorem gives an effective method that allow us to rewrite every term in $\t{m,cyl_n}{}$ as a disjunction of normal forms of the  same degree.
A general version of this theorem is proved in the preprint \cite{amn}.
Also, some similar versions of this theorem were proved in \cite{fine} and \cite{me}.
\begin{theorem}\label{andreka}Let $k\in\omega$. Then the followings are true:
\begin{enumerate}[(i)]
\item\label{andc1} $K^{'}\models\sum F^{n,m}_k=1$.
\item\label{andc2} For every $\tau,\sigma\in F^{n,m}_k$, if $\tau\not=\sigma$ then $K^{'}\models\tau\cdot \sigma=0$.
\item\label{andc4} There exists an effective method (a finite algorithm) to find, for every $\tau\in\t{m,cyl_n}{}$, a non-negative integer $q\in\omega$ and a finite set $S_{\tau}\subseteq F^{n,m}_q$ such that $K^{'}\models\tau=\sum S_{\tau}$.
\end{enumerate}
\end{theorem}
Therefore, it is enough to prove that every satisfiable normal form in a member of $K$ is satisfiable in a finite structure member of $K$. Then by theorem \ref{andreka} (\ref{andc4}), one can easily conclude that $K$ is generated by its finite members. The finite algebras we construct here are the complex algebras of some atom structures.
\subsection{Atom Structures}
Let $cat_{n}$ be the relational type with binary relations $T_i$ and unary relations $E_{ij}$, $i,j\in n$.
\begin{definition}
Let $\mathbf{S}=\langle S, T_i, E_{ij}\rangle_{i,j\in n}$ be a model of type $cat_{n}$. The complex algebra over $S$ is defined as follows.
$$\mathfrak{Cm}(\mathbf{S})=\langle \mathcal{P}(S), \cup,\cap,\setminus,\emptyset, S, T^{\star}_i, E_{ij}^{\star}\rangle_{i,j\in n}\in CTA_{n}.$$
Where, $\mathcal{P}(S)$ is the power set of $S$, $\cup,\cap,\setminus$ are the usual boolean set operations and, for every $i,j\in n$, $E_{ij}^{\star}=E_{ij}$ and $T_i^{\star}X=\{y\in S:\exists x\in X\text{ and }(x,y)\in T_i\}$.
\end{definition}
While building the atoms structures, we need to restrict ourself to some conditions to guarantee that the complex algebras of these atom structures are as desired. Consider a model of $cat_{n}$, $\mathbf{S}=\langle S, T_i, E_{ij}\rangle_{i,j\in n}$. One can easily check that $\mathfrak{Cm}(\mathbf{S})\in NCA_{n}$ if and only if $\mathbf{S}\models\{AS1,AS2,AS3,AS5\}$ and $\mathfrak{Cm}(\mathbf{S})\in WCA_{n}$ if and only if $\mathbf{S}\models\{AS1,AS2,AS4,AS5\}$. Where,
\begin{eqnarray*}
AS1 &=& \{T_i\text{ is an equivalence relation on }S: i\in n\}.\\
AS2 &=& \{E_{ii}=S:i\in n\}.\\
AS3 &=& \{E_{ij}=T_k^{\star}(E_{ik}\cap E_{kj}): i,j,k\in n\text{ and }k\not\in\{i,j\}\}.\\
AS4 &=& \{E_{ij}=E_{ji}, E_{ik}\cap E_{kj}\subseteq E_{ij}, E_{ij}=T_k^{\star}E_{ij}: i,j,k\in n\}.\\
AS5 &=& \{T_i\cap {^{2}{E_{ij}}}\subseteq Id : i,j\in n\text{ and }i\not=j\}.
\end{eqnarray*}

Fix a finite number $k\in\omega$ and a normal form $\tau\in F_k^{n,m}$. We construct a finite atom structure $S^{\tau,K}$ whose complex algebra decides the satisfiability of $\tau$. We note that if the term $\tau$ is satisfiable in $K$, then the syntactical construction of $\tau$ guarantees that $\a{Cm}(S^{\tau,K})\in WCA_n$. For $K=NCA_n$, we need to be a little bit careful while considering the forms of degree $0$. Indeed, the syntactical constructions of the forms of degree $0$ are not enough to guarantee condition $AS3$.
\section{Decidability and finite algebra property}
Recall that every form in $F^{n,m}$ is determined by some information given on the diagonals, free variables and the cylindrifications of the forms of the first smaller degree. So, we need to introduce some notions that allow us to handle these information easily.
\begin{definition}
Let $i<n$, $k\in\omega$, $\beta\in{^{D_{n,m}}\{-1,1\}}$ and $\alpha\in{^{CF_k^{n,m}}\{-1,1\}}$. Define
\begin{center}
$sub_i(D_{n,m}^{\beta}\cdot(CF_k^{n,m})^{\alpha}):=\{\sigma\in F_k^{n,m}:\alpha(c_i\sigma)=1\}$, and
\end{center}
\begin{center}$color(D_{n,m}^{\beta}):=color(D_{n,m}^{\beta}\cdot(CF_k^{n,m})^{\alpha}):=\{\sigma\in D_{n,m}:\beta(\sigma)=1\}$.\end{center}
\end{definition}
Recall that $k\in\omega$ is fixed and $\tau\in F_k^{n,m}$. We construct the structure $S^{\tau,K}$ inductively. Pick up a node $u$ and assign the label $L(u)=\tau$. Set $S_0=S_0^0=\cdots=S_0^{n-1}=\{u\}$. Set $T_0^0=\cdots=T_0^{n-1}=\{(u,u)\}$. The labeling assigned to the node $u$ is to indicate that $u$ is responsible for satisfying $\tau$ at the end of the construction. To guarantee this, we need to extend $S_0$ by the information given by $sub_i(\tau)$, $i<n$, as follows. Let $U$ be an infinite set that doesn't contain $u$. For each $i<n$, construct an injective function,
$$\psi_u^i:\{\sigma\in sub_i(\tau):(\forall j\in n\setminus\{i\}) \text{ } d_{ij}\not\in color(\sigma)\cap color(\tau)\}\longrightarrow U,$$ such that $Rng(\psi_u^i)$'s are pairwise disjoint and $U\setminus(Rng(\psi_u^0)\cup\cdots Rng(\psi_u^{n-1}))$ is infinite. For every $i\in n$, set $S^i_{1}=Rng(\psi_u^i)$ and set $S_{1}=S_1^0\cup\cdots\cup S_1^{n-1}$. We extend the labels as follows. For every $i\in n$ and every $v\in Rng(\psi^i_u)$, define $L(v):=(\psi^i_u)^{-1}(v)$. It remains to define $T^i_{1}=\{(v,w):v,w\in Rng(\psi_u^i)\cup\{u\}\}$ for every $i\in n$.

Note that the terms in $sub_i$ whose colors share $d_{ij}$ with $color(\tau)$ (for some $j\in n\setminus\{i\}$) were omitted because $C_7$ implies that $d_{ij}\cdot c_i(d_{ij}\cdot x)\leq x$. It is not hard to see that, under a suitable evaluation, if every element in $S_1$ satisfies its label then $u$ satisfies $\tau$. Hence, we need to extend $S_1$ by adding more elements according to the information carried by the functions $sub_i$, $i<n$, to guarantee that each element of $S_1$ satisfies its label. Let $i\in n$. Note that every two elements $v,w\in S_1^i$ have to be $i$-connected because $T^i_1$ is transitive. Therefore, the information given by $sub_i$ is already guaranteed by the elements in $S_1^i$. So, for the nodes in $S_1^i$ we need to consider the information given by $sub_j$, $j\not=i$, only.

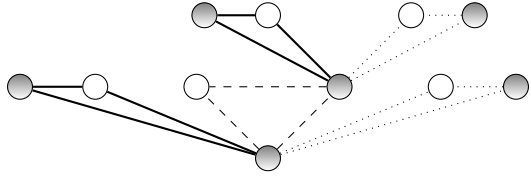
\begin{figure}[!h]
\centering
\begin{tikzpicture}[>=latex'][center]
\node [circle, draw, shade, black] (a) {};
\node [circle, draw, shade, black] (b) [above right=1cm of a]{};
\node [circle, draw, black] (x) [above left=1cm of a] {};
\node [circle, draw, black] (y) [above right=1cm of b] {};
\node [circle, draw, shade, black] (w) [right=0.5cm of y] {};
\node [circle, draw, black] (v) [above left=1cm of b] {};
\node [circle, draw, shade, black] (c) [left=0.5cm of v] {};
\draw [-][dashed] (a) -- (b);
\node [circle, draw, black] (d) [left=1cm of x] {};
\node [circle, draw, shade, black] (h) [left=2cm of x] {};
\node [circle, draw, black] (q) [right=1cm of b] {};
\node [circle, draw, shade, black] (p) [right=2cm of b] {};
\draw [-][thick] (a) -- (h);
\draw [-][thick] (a) -- (d);
\draw [-][thick] (h) -- (d);
\draw [-][dotted] (w) -- (y);
\draw [-][dotted] (b) -- (w);
\draw [-][dotted] (b) -- (y);
\draw [-][thick] (b) -- (v);
\draw [-][thick] (b) -- (c);
\draw [-][thick] (c) -- (v);
\draw [-][dashed] (b) -- (x);
\draw [-][dashed] (a) -- (x);
\draw [-][dotted] (a) -- (q);
\draw [-][dotted] (a) -- (p);
\draw [-][dotted] (p) -- (q);
\end{tikzpicture}
\caption{The relation structure $S^{\tau,K}$}
\label{11}
\end{figure}

More generally, suppose that $k\geq 2$ and $S_l$, $S_l^0,\ldots,S_l^{n-1}$ have been constructed and the labeling $L$ has been extended to cover $S_l$, for some $1\leq l\leq k-1$. For every $i\in n$ and every $v\in S_l\setminus S_l^i$, create an injective function $$\psi_v^i:\{\sigma\in sub_i(L(v)):(\forall m\in n\setminus\{i\}) \text{ } d_{im}\not\in color(\sigma)\cap color(L(v))\}\longrightarrow U\setminus(S_0\cup\cdots\cup S_l),$$
Such that the ranges of all those functions $\psi_v^i$'s are pairwise disjoint and the set $U\setminus(S_0\cup\cdots\cup S_l\cup S_{l+1})$ is still infinite. Where, $S_{l+1}=S_{l+1}^0\cup\cdots\cup S_{l+1}^{n-1}$ and, for every $i\in n$, $S_{l+1}^i=\bigcup\{Rng(\psi_v^i):v\in S_l\setminus S_l^i\}$. We extend the labels as expected. For every $i\in n$, every $v\in S_l\setminus S_l^i$ and every $w\in Rng(\psi^i_v)$, define $L(w):=(\psi^i_v)^{-1}(w)$. Finally, for every $i\in n$, define $T_{l+1}^i=\{(w_1,w_2):w_1,w_2\in Rng(\psi_v^i)\cup\{v\},v\in S_l\setminus S_l^i\}$.

We are almost done, it remains to add some extra nodes to guarantee that $S^{\tau,K}\models AS3$ if $K=NCA_n$ and $\tau$ is satisfiable in $NCA_n$. First, we introduce some notations. For any two sequences $f,g$ of length $n$ and for any $i\in n$, we write $g\equiv_ih$ if and only if $g(j)=h(j)$ for every $j\in n\setminus\{i\}$. Let $f$ be any  sequence of length $n$ and let $i,j,k\in n$. If $k\not\in\{i,j\}$ and $f(i)=f(j)$, then define $C_k^{i,j}f$ to be the sequence which is like $f$ except that its value at $k$ equals to $f(i)$. Otherwise, define $C_k^{i,j}f:=f$.

For every node $v\in S_k$, we say that $v$ is representable if and only if there exists $f=(f_0,\cdots,f_{n-1})$ such that $\mid\{f_0,\ldots,f_{n-1}\}\mid\leq n-1$ and $(\forall i,j\in n) \text{ } (f_i=f_j\iff d_{ij}\in color(L(v)))$. For every representable node $v\in S_k$, pick up such a representation tuple $f_v$ such that $Rng(f_v)\cap Rng(f_w)=\emptyset$ if $v\not=w$. We add some extra part to $S$ as follows. For every representable $v\in S_k$, define $Rep(v)$ to be the smallest subset of $^nRng(f_v)$ that is closed under the operations $C_k^{i,j}$, for every $i,j,k\in n$, and contains the element $f_v$.

If $K=NCA_n$, set $S_{-1}=\bigcup\{Rep(v)\setminus\{f_v\}:v\in S_k\text{ is representable and }K=NCA_n\}$ and
\begin{eqnarray*}
T^{i}_{-1}&=&\{(g,h):g,h\in (Rep(v)\setminus\{f_v\}), v\in S_{k}\text{ is representable and }g\equiv_ih\}\\
&=&\{(g,v):g\in (Rep(v)\setminus\{f_v\}), v\in (S_{k}\setminus S_k^i)\text{ is representable and }g\equiv_if_v\}\\
&=&\{(v,g):g\in (Rep(v)\setminus\{f_v\}), v\in (S_{k}\setminus S_k^i)\text{ is representable and }g\equiv_if_v\},
\end{eqnarray*}
for every $i\in n$. If $K=WCA_n$, set $S_{-1}=T^{0}_{-1}=\cdots=T^{n-1}_{-1}=\emptyset$. We don't need to extend the labels. Define the desired atom structure $S^{\tau,K}$ as follows. $$S^{\tau,K}=\langle S, T_i,E_{ij}\rangle_{i,j\in n},$$
where, for every $i,j\in n$, $S=S_{-1}\cup S_0\cup\cdots\cup S_{n-1}$, $T_i=T^i_{-1}\cup T^i_0\cup\cdots\cup T^i_{k}$ and
$$E_{ij}=\{v\in S\setminus S_{-1}: d_{ij}\in color(L(v))\}\cup \{g\in S_{-1}:g(i)=g(j)\}.$$
Define the evaluation $e^{\tau}:\{x_0,\ldots,x_{m-1}\}\rightarrow\mathcal{P}(S)$ as follows. For every $0\leq i\leq m-1$, let $e^{\tau}(x_i):=\{v\in S\setminus S_{-1}:x_i\in color(L(v))\}$.
\begin{lemma}\label{graph}$\f{m}{K}\models\tau\not=0$ if and only if $\mathfrak{Cm}(S^{\tau,K})\in K$ and
$(\mathfrak{Cm}(S^{\tau,K}),e^{\tau},v)\models L(v)$,
for every $v\in S\setminus S_{-1}$.
\end{lemma}
\begin{proof}
We prove the non trivial direction. Suppose that $\f{m}{K}\models\tau\not=0$ then $\f{m}{K}\models L(v)\not=0$ for every $v\in S\setminus S_{-1}$. Therefore, by the construction of $S^{\tau,K}$, $S_{-1}$ and $T_{-1}$, one can see that $\mathfrak{Cm}(S^{\tau,K})\in K$. Let $0\leq l\leq k$ and suppose $v\in S_k$, then $L(v)\in F^{n,m}_{k-l}$. Let $0\leq h\leq k$. If $h\geq k-l$, define $tag_h(v)=L(v)$. Otherwise, since $\f{m}{K}\models L(v)\not=0$, there exists unique $\sigma\in F_h^{n,m}$ such that $\f{m}{K}\models L(v)\leq\sigma$. In this case, define $tag_h(v)=\sigma$. It is enough to prove the following. For every $0\leq h\leq k$ and every $v\in S\setminus S_{-1}$.
\begin{equation}\label{tag}
(\mathfrak{Cm}(S^{\tau,K}),e^{\tau},v)\models tag_h(v).
\end{equation}
We use induction on $h$. From the construction of $\mathfrak{Cm}(S^{\tau,K})$ and the choice of the evaluation $e^{\tau}$, it is clear that $\mod{v}{tag_0(v)}$, for every $v\in S\setminus S_{-1}$. Let $0\leq h\leq k-1$ and assume that for every $v\in S\setminus S_{-1}$,
$(\mathfrak{Cm}(S^{\tau,K}),e^{\tau},v)\models tag_h(v)$. Let $0\leq l\leq k$, $i\in n$ and $v\in S_l^i$. If $h+1\geq k-l$, then $(\mathfrak{Cm}(S^{\tau,K}),e^{\tau},v)\models tag_{h+1}(v)=tag_{h}(v)$. Suppose that $h+1< k-l$, then $tag_{h+1}(v)\in F^{n,m}_{h+1}$. Let $j\in n$ be such that $j\not=i$ and let $\sigma\in F_h^{n,m}$.
\begin{itemize}
\item Suppose that $\sigma\in sub_j(tag_{h+1}(v))$. Since $\f{m}{K}\models L(v)\leq tag_{h+1}(v)$, then there exists $\sigma^{'}\in sub_j(L(v))\subseteq F_{k-l-1}^{n,m}$ such that $\f{m}{K}\models \sigma^{'}\leq\sigma$. Suppose that there is no $m\in n\setminus\{j\}$ with $d_{jm}\not\in color(\sigma^{'})\cap color(L(v))$. Then, by the construction of $S^{\tau,K}$, there exists $w\in S_{l+1}^j$ such that $(v,w)\in T_j$ and $L(w)=\sigma^{'}$. By induction, we have $(\mathfrak{Cm}(S^{\tau,K}),e^{\tau},w)\models tag_h(w)=\sigma$. Consequently, $(\mathfrak{Cm}(S^{\tau,K}),e^{\tau},v)\models c_j\sigma$. Suppose that there exists $m\in n\setminus\{j\}$ such that $d_{jm}\in color(\sigma^{'})\cap color(L(v))$. Hence, by the  axiom $C_7$, $\f{m}{K}\models L(v)=L(v)\cdot d_{jm}\leq c_j(\sigma^{'}\cdot d_{jm})\cdot d_{jm}\leq \sigma^{'}\cdot d_{jm}=\sigma^{'}$. By induction we have $(\mathfrak{Cm}(S^{\tau,K}),e^{\tau},v)\models tag_h^{\tau}(v)=\sigma$ and consequently $(\mathfrak{Cm}(S^{\tau,K}),e^{\tau},v)\models c_i\sigma$.

\item Suppose that $\sigma\not\in sub_j(tag_{h+1}(v))$. Assume toward a contradiction that there exists a node $w\in S$ such that $(v,w)\in T_j$ and $(\mathfrak{Cm}(S^{\tau,K}),e^{\tau},w)\models\sigma$. If $w=v$, then we should have $\sigma=tag_h(v)$. Which makes a contradiction with the assumption that $\sigma\not\in sub_j(tag_{h+1}(v))$ and the fact that $\f{m}{K}\models L(v)\not=0$. If $w\not=v$ then, by the construction of $S^{\tau,K}$, there exists $\sigma^{'}\in sub_j(L(v))\subseteq F_{k-l-1}^{n,m}$ such that $L(w)=\sigma^{'}$. Therefore, $\f{m}{K}\models \sigma^{'}\leq \sigma$. Since $\f{m}{K}\models L(v)\leq tag_{h+1}(v)$ and $\sigma^{'}\in sub_j(L(v))$, then $\sigma\in sub_j(tag_{h+1}(v))$. Which contradicts the assumption.
\end{itemize}
Therefore, for every $j\in n\setminus\{i\}$ and every $\sigma\in F_h^{n,m}$,
$$(\mathfrak{Cm}(S^{\tau,K}),e^{\tau},v)\models c_j\sigma\iff\sigma\in sub_j(tag_{h+1}(v)).$$
If $l=0$, then by the same argument above one can easily show that, for every every $\sigma\in F_h^{n,m}$, $(\mathfrak{Cm}(S^{\tau,K}),e^{\tau},v)\models c_i\sigma\iff\sigma\in sub_i(tag_{h+1}(v))$. So assume that $l\not=0$ and let $\sigma\in F_h^{n,m}$ be such that $\sigma\in sub_i(tag_{h+1}(v))$. Then there exists $\sigma^{'}\in sub_i(L(v))$ such that $\f{m}{K}\models\sigma^{'}\leq\sigma$. Let $w$ be the unique node in $S_{l-1}$ with $(w,v)\in T_i$. Since $\f{m}{K}\models L(w)\not=0$, then there exists $\sigma^{''}\in sub_i(L(w))$ such that $\f{m}{K}\models\sigma^{''}\leq \sigma^{'}$. But by the construction there exists a node $z\in S_l^i$ (not necessarily different than $v$) such that $\{(w,z),(v,z)\}\subseteq T_i$ and $L(z)=\sigma^{''}$. By induction we have $(\mathfrak{Cm}(S^{\tau,K}),e^{\tau},z)\models tag_h(z)=\sigma$. Consequently, $(\mathfrak{Cm}(S^{\tau,K}),e^{\tau},v)\models c_i\sigma$. Using this idea and the above method one can prove that for every $\sigma\in F_h^{n,m}$,
$$(\mathfrak{Cm}(S^{\tau,K}),e^{\tau},v)\models c_i\sigma\iff\sigma\in sub_i(tag_{h+1}(v)).$$
By the induction hypothesis, we have $(\mathfrak{Cm}(S^{\tau,K}),e^{\tau},v)\models tag_0(v)$. Therefore, by the  induction principle (\ref{tag}) is true for every $0\leq h\leq k$ and every $v\in S\setminus S_{-1}$, as desired.
\end{proof}
Thus, we have shown that every satisfiable normal form in $K$ has a finite witness in $K$. Therefore, it follows that the equational theory of the class $K$ is decidable and coincide with the equational theory of the finite members of $K$.
\begin{theorem}
The class $K$ is generated by its finite members and has a decidable equational theory.
\end{theorem}
\begin{proof}
Recall that $m\geq 0$ is arbitrary but fixed. Therefore, to show that $K$ is generated by its finite members, we need to show the following. For every $\sigma\in\t{m,cyl_n}{}$,
\begin{equation*}
K\not\models\sigma=0 \Longrightarrow \text{ } (\exists\text{ a finite algebra }\a{A}\in K) \text{ } \text{ } \a{A}\models\sigma\not=0.
\end{equation*}
Let $\sigma\in\t{m,cyl_n}{}$ and suppose that $K\not\models\sigma=0$. Then, by theorem \ref{andreka}(\ref{andc4}), there exists a non-negative $k\in\omega$ and $\tau\in F_k^{n,m}$ such that $K\models 0\not=\tau\leq\sigma$. Consider the atom structure $S^{\tau,K}$ defined above. Note that its complex algebra, $\a{Cm}(S^{\tau,K})$, is a finite member of $K$. By lemma \ref{graph}, we have $\a{Cm}(S^{\tau,K})\models\tau\not=0$. Hence, $\a{Cm}(S^{\tau,K})\models\sigma\not=0$. Therefore, $K$ is generated by its finite members. For the decidability, we need to find a finite algorithm that decides, for every $\sigma_1,\sigma_2\in\t{m,cyl_n}{}$, whether $K\models\sigma_1=\sigma_2$ or not. By the fact that,
$$(\forall\sigma_1,\sigma_2\in\t{m,cyl_n}{}) \text{ } \text{ } K\models\sigma_1=\sigma_2\Longleftrightarrow K\models(\sigma_1\cdot-\sigma_2+-\sigma_1\cdot\sigma_2)=0,$$
it is enough to find a finite algorithm that decides whether $K\models\tau=0$ or not, for every $\tau\in\t{m,cyl_n}$. The proof of lemma \ref{graph} can be translated to a finite algorithm that decides $K\models\tau=0$, for every $\tau\in F^{n,m}$. The desired algorithm is the combination of this algorithm with the finite algorithm given in theorem \ref{andreka} (\ref{andc4}) and an algorithm decides which node in $S_k$ is representable.
\end{proof}
We note that decidability also follows without referring to the actual construction from finite axiomatization and finite model property,
as follows. The equational theory is recursively enumerable by finite axiomatizability, and the non-equations are also recursively
enumerable by finite model property.
\section{Non-atomicity of the free algebras}
From the universal algebra, the free algebras of a variety play an essential role in understanding this variety. In algebraic logic, the free algebras of a variety corresponding to a logic $\mathcal{L}$ are even more important. Indeed, they correspond to the Lindenbaum-Tarski algebras of the logic $\mathcal{L}$. Atoms in the Lindenbaum-Tarski algebras of sentences correspond to finitely axiomatizable complete theories. Thus, non-atomicity of the free algebras is equivalent to weak G\"odel's incompleteness property of the corresponding logic. In this section, we prove the following.
\begin{theorem}In the free algebra $\f{m}{K}$, there is no atom below $t=\prod\{-d_{ij}:i<j<n\}$. Thus, $\f{m}{K}$ is not atomic.
\end{theorem}
\begin{proof}
By theorem \ref{andreka} (\ref{andc4}), it is enough to show that every satisfiable form that is below $t$ in the free algebra $\f{m}{K}$ is not an atom. Fix a finite number $k\in\omega$ and let $\tau\in F_k^{n,m}$ be such that $\f{m}{K}\models 0\not=\tau\leq t$. Recall the structure $S^{\tau,K}$ constructed in the previous section. In few steps, we use $S^{\tau,K}$ to achieve our aim
\begin{description}
\item[Step 1:] We construct a sequences $v_0,\ldots,v_k\in S$ of nodes satisfying the following conditions.
\begin{enumerate}
\item $L(v_q)\in F_{k-q}^{n,m}$ for every $q\in k+1$. In particular, $L(v_0)=\tau$.
\item For every $q\in k$: If $q$ is even then $(v_q,v_{q+1})\in T_0$. If $q$ is odd then $(v_q,v_{q+1})\in T_1$.
\end{enumerate}

To start, let $v_0=u\in S_0$ be the unique node in $S_0$ whose label is $\tau$. If $k=0$, then we are done. Suppose that $k\geq 1$, since $\f{m}{K}\models\tau\not=0$ then there exists unique $\tau_1\in F_{k-1}^{n,m}$ such that $\f{m}{K}\models\tau\leq \tau_1$. Hence, $\tau_1\in sub_0(\tau)$. But, for every $i,j\in n$, if $i\not=j$ then $d_{ij}\not\in color(\tau_1)\cap color(\tau)$. Therefore, by the construction of $S_1$, there exists a unique node $v_1\in S_1$ that satisfies $L(v_1)=\tau_1$ and $(v_0,v_1)\in T_0$. If $k=1$, then we are done. Suppose that $k\geq 2$, since $\f{m}{K}\models\tau_1\not=0$ then there exists unique $\tau_2\in F_{k-2}^{n,m}$ such that $\f{m}{K}\models\tau_1\leq\tau_2$. Consequently, $\tau_2\in sub_1(\tau_1)$ and $d_{ij}\not\in color(\tau_1)\cap color(\tau_2)$, for every $i,j\in n$, $i\not=j$. Let $v_2$ be the unique node in $S_2$ with $L(v_2)=\tau_2$ and $(v_1,v_2)\in T_1$. Continue in this manner, we get the desired sequence.
\item[Step 2:] We extend $S$ to $S_+$ as follows. Without loss of generality we may assume that $k$ is even (if $k$ is odd we can easily modify  what follows to be suitable). Let $f_0,\ldots,f_{n-1}$ be mutually different elements such that each of which is different than all the nodes in $S$. Identify $f:=(f_0,\ldots,f_{n-1})$ with the node $v_k$. Define $Rep(v_k)$ to be the smallest subset of ${^n\{f_0,\ldots,f_{n-1}\}}$ that is closed under the operations $C_k^{i,j}$, $i,j,k\in n$, and contains the element $(f_1,f_1,f_2,\ldots,f_{n-1})$. Define $S^{\tau,K}_+=\langle S_+,T_i^+,E_{ij}^+\rangle_{i,j\in n}$, where
    \begin{eqnarray*}
    S_+&=&S\cup(Rep(v_k)\setminus\{f\})\\
    T_0^+&=&T_0\cup\{(g,h)\in{^2(Rep(v_k)\setminus\{f\})}:g\equiv_0 h\}\\
    T_i^+&=&T_i\cup\{(g,h)\in{^2Rep(v_k)}:g\equiv_i h\},\text{ for every }i\in n\setminus\{0\}\\
    E_{ij}^+&=&E_{ij}\cup\{g\in Rep(v_k):g(i)=g(j)\},\text{ for every }i,j\in n.
    \end{eqnarray*}
    Recall the evaluation $e^{\tau}$ defined with $\a{Cm}(S^{\tau,K})$. By a similar argument to the proof of lemma \ref{graph}, one can see that $\a{Cm}(S_+^{\tau,K})\in K$ and
    $$(\forall v\in S\setminus S_{-1}) \text{ } \text{ } \text{ } (\a{Cm}(S_+^{\tau,K}),e^{\tau},v)\models L(v).$$
\item[Step 3:] For every $q\in k+1$, by theorem \ref{andreka} (\ref{andc1},\ref{andc2}), $F^{n,m}_{k-q+1}$ forms a partition of the unit then there exist unique terms $\sigma_q,\gamma_q\in F_{k-q+1}^{n,m}$ such that $(\a{Cm}(S^{\tau,K}),e^{\tau},v_q)\models \sigma_q$ and $(\a{Cm}(S_+^{\tau,K}),e^{\tau},v_q)\models \gamma_q$. Note that $L(v_q)\in F_{k-q}^{n,m}$. Therefore,
    \begin{equation}\label{seq}\f{m}{K}\models0\not=\sigma_q\leq L(v_q)\text{ } \text{ } \text{ and }\text{ } \text{ } \f{m}{K}\models0\not=\gamma_q\leq L(v_q).\end{equation}
\item[Step 4:] We prove the following. For every $q\in k+1$, $\f{m}{K}\models\sigma_q\cdot\gamma_q=0$.
We use induction on $k-q$. By the extension $Rep(v_k)$, it is clear that $\f{m}{K}\models\sigma_k\leq -c_0d_{01}$ but $\f{m}{K}\models\gamma_k\leq c_0d_{01}$. Hence, $\f{m}{K}\models\sigma_k\cdot\gamma_k=0$. The induction step is going in a similar way. Suppose that $\f{m}{K}\models\sigma_{k-q}\cdot\gamma_{k-q}=0$, for some $q\in k$. Let $i<2$ be such that $i=k-q-1$ $(mod\text{ }2)$. Remember that $(v_{k-q-1},v_{k-q})\in T_i$. Also, remember that $\sigma_{k-q},\gamma_{k-q},L(v_{k-q-1})\in F_{q+1}^{n,m}$. But $\sigma_{k-q}$ and $\gamma_{k-q}$ are disjoint in $\f{m}{K}$, by the induction hypothesis. Hence at least one of $\sigma_{k-q}$ and $\gamma_{k-q}$ is disjoint from $L(v_{k-q-1})$. Without loss of generality, we may assume that $\f{m}{K}\models \sigma_{k-q}\cdot L(v_{k-q-1})=0$. By the construction of $S$ we have, for every node $v\in S\setminus\{v_{k-q-1},v_{k-q}\}$, if $(v,v_{k-q-1})\in T_i$ then $\f{m}{K}\models L(v)\cdot L(v_{k-q})=0$. Therefore, for every $v\in S\setminus\{v_{k-q}\}$, if $(v,v_{k-q-1})\in T_i$ then $(\a{Cm}(S^{\tau,K}),e^{\tau},v)\not\models\sigma_{k-q}$ and $(\a{Cm}(S_+^{\tau,K}),e^{\tau},v)\not\models\sigma_{k-q}$. Remember that $\sigma_{k-q}$ and $\gamma_{k-q}$ were chosen such that $(\a{Cm}(S^{\tau,K}),e^{\tau},v_{k-q})\models\sigma_{k-q}$ and $(\a{Cm}(S_+^{\tau,K}),e^{\tau},v_{k-q})\models\gamma_{k-q}$. Hence,
$$(\a{Cm}(S_+^{\tau,K}),e^{\tau},v_{k-q-1})\models\gamma_{k-q-1}\cdot-c_i\sigma_{k-q}$$ and $$(\a{Cm}(S^{\tau,K}),e^{\tau}, v_{k-q-1})\models\sigma_{k-q-1}\cdot c_i\sigma_{k-q}.$$
Therefore, $\f{m}{K}\models\sigma_{k-q-1}\leq c_i\sigma_{k-q}$ but $\f{m}{K}\models\gamma_{k-q-1}\leq-c_i\sigma_{k-q}$. In other words, $\f{m}{K}\models\sigma_{k-q-1}\cdot\gamma_{k-q-1}=0$. Hence, we are done by the induction principle.
\end{description}
We have shown that there exist two forms $\sigma_0,\gamma_0$ each of which is satisfiable form below $\tau$ inside the free algebra $\f{m}{K}$. We also proved that $\f{m}{K}\models\sigma_0\cdot\gamma_0=0$. Therefore, $\tau$ is not an atom in $\f{m}{K}$ as desired.
\end{proof}
\section*{Acknowledgment} The results in this paper are parts of the authors PhD research, under the supervision of professors Hajnal Andr\'eka and Istv\'an N\'emeti. The author shall like to thank professors Hajnal Andr\'eka and Istv\'an N\'emeti for their great help and valuable comments that made the paper in its present form.
\bibliographystyle{plain}

\begin{thebibliography}{100}
\bibitem{low15}
L. L\"owenheim (1915).
\newblock {\"U}ber M\"oglichkeiten im Relativkalkul
\newblock Mathematische Annalen, 76, pp. 447--470.
\bibitem{anderson}
A. R. Anderson (1954).
\newblock Improved decision procedures for {L}ewis's calculus $S4$ and von {W}right's calculus $M$.
\newblock The Journal of Symbolic Logic, 19 (03), pp. 201-214.
\bibitem{ca1}
L. Henkin and J. D. Monk and A. Tarski (1971).
\newblock Cylindric Algebras, part 1.
\newblock North-Holland publishing company, Studies in Logic and the Foundation of Mathematics, 64.
\bibitem{fine}
K. Fine (1975).
\newblock Normal forms in modal logic.
\newblock Notre Dame journal of formal logic, 16 (12), pp. 229-237.
\bibitem{prenem85}
I. Nemeti (1985).
\newblock Logic with three variables has {G}{\"o}del incompleteness propoerty-thus free cylindric algebras are not atomic.
\newblock Math. Inst. Budapest, Preprint, http://www.renyi.hu/~nemeti/NDis/NPrep85.pdf.
\bibitem{nem86}
I. Nemeti (1986).
\newblock Free algebras and decidability in algebraic logic.
\newblock Academic Doctoral Dissertation, Hungarian Academy of Sciences, Budapest.
\bibitem{al}
H. Andr{\'e}ka, J. D. Monk and I. N{\'e}meti (1991).
\newblock Algebraic Logic.
\newblock North Holland, Amsterdam, Colloquia Mathematica Societatis J{\'a}nos Bolyai, 54.
\bibitem{vB95}
J. van Benthem (1994).
\newblock A note on dynamic arrow logic.
\newblock In Logic and Information Flow. Eds: J. van Eijck and A. Visser.
\newblock The MIT Press, pp. 15-29.
\bibitem{nem95}
I. Nemeti (1995).
\newblock Decidable versions of first order logic and cylindric-relativized set algebras.
\newblock In: Logic Colloquium'92 (Proc. Veszprem, Hungary 1992), eds: L. Csirmaz and D. M. Gabbay and M. de Rijke,
\newblock Studies in Logic, Language and Computation, CSLI Publications, pp.177-241.
\bibitem{gyn}
Z. Gyenis (2011).
\newblock On atomicity of free algebras in certain cylindric-like varieties.
\newblock Logic Journal of IGPL, 19 (1), pp. 44-52.
\bibitem{ca4}
H. Andr\'eka, M. Ferenczi and I. N\'emeti (2013).
\newblock Cylindric-like algebras and algebraic logic.
\newblock Springer Publishing Company, 22.
\bibitem{me}
M. Khaled (2015).
\newblock Weak G\"odel's incompleteness property for some decidable versions of the calculus of relations.
\newblock Review of Symbolic Logic, Submitted.
\bibitem{amn}
M. Khaled (2015).
\newblock Disjunctive normal forms for any class of Boolean algebras with operators.
\newblock arXiv:1511.03631 [math.LO].
\end{thebibliography}

\end{document}